\documentclass[12pt]{article}
\usepackage{e-jc}
\specs{P3.5}{23(3)}{2016}

\usepackage{amsthm, amsmath, amssymb}
\usepackage{enumerate}
\usepackage{ifthen}
\usepackage{calc}
\usepackage{graphicx}
\usepackage[margin=40pt]{caption}
\usepackage{subcaption}
\usepackage{microtype}
\usepackage{tikz}
\usetikzlibrary{calc}

\usepackage[colorlinks=true,citecolor=black,linkcolor=black,urlcolor=blue]{hyperref}
 \newcommand{\arxiv}[1]{\href{http://arxiv.org/abs/#1}{\texttt{arXiv:#1}}}

\theoremstyle{plain}
\newtheorem{theorem}{Theorem}[section]

\newtheorem{lemma}[theorem]{Lemma}

\title{\bf Generalizing the divisibility property\\ of rectangle domino tilings}

\author{Forest Tong\\
\small Department of Mathematics\\[-0.8ex]
\small Massachusetts Institute of Technology\\[-0.8ex]
\small Cambridge, MA, USA \\
\small\tt forest1218us@gmail.com
}

\date{\dateline{Jul 8, 2014}{Jun 28, 2016}{Jul 8, 2016} \\
\small Mathematics Subject Classifications: 05C70, 05C50, 05C30, 05C22}

\DeclareMathOperator{\s}{sgn}
\DeclareMathOperator{\w}{wt}
\newcommand{\swh}{\overline{H}}
\newcommand{\G}{\mathcal G}
\newcommand{\C}{\mathcal C}
\newcommand{\Rd}{R^{\prime\prime}}
\newcommand{\K}[1]{K^{(#1)}}

\begin{document}

\maketitle

\begin{abstract}
We introduce a class of graphs called \emph{compound graphs}, generalizing rectangles, which are constructed out of copies of a planar bipartite \emph{base graph}. The main result is that the number of perfect matchings of every compound graph is divisible by the number of matchings of its base graph. Our approach is to use Kasteleyn's theorem to prove a key lemma, from which the divisibility theorem follows combinatorially. This theorem is then applied to provide a proof of Problem 21 of Propp's \emph{Enumeration of Matchings}, a divisibility property of rectangles. Finally, we present a new proof, in the same spirit, of Ciucu's factorization theorem.
\end{abstract}

\begin{section}{Introduction}

The number of perfect matchings of a planar graph (or equivalently, the number of domino tilings of its planar dual, as shown in Figure~\ref{fig:planar}) has been the subject of much study, and many striking numerical relations have been observed among the matchings of related graphs.
One such relation, which provides the motivation for this paper, is the following problem posed by Jim Propp \cite{propp99}. Let $R(m, n)$ denote the $m \times n$ rectangular graph, or, more formally, $P_m \times P_n$, where $P_k$ is the path graph with $k$ vertices. Let $\#G$ denote the number of perfect matchings (which we will refer to simply as ``matchings") of a graph $G$. The divisibility property referred to in the title is that if $a, b, A, B$ are positive integers such that $a + 1 \mid A + 1$ and $b + 1 \mid B + 1$, where $\mid$ represents integer division, then $\#R(a, b) \mid \#R(A, B).$

\begin{figure}[tb!]
\centering
\begin{subfigure}{0.3\textwidth}
\begin{tikzpicture}[scale=.7]
\draw[fill=gray] (0, 0) grid (4, 3);

\foreach \p in {(0, 0), (2, 0), (1, 1), (1, 2), (0, 3), (2, 3)}
\draw[ultra thick] \p -- +(1, 0);

\foreach \p in {(0, 1), (3, 1), (4, 0), (4, 2)}
\draw[ultra thick] \p -- +(0, 1);

\foreach \x in {0,...,4}
\foreach \y in {0,...,3}
\draw[fill=black] (\x, \y) circle (2pt);

\end{tikzpicture}
\end{subfigure}
\begin{subfigure}{0.3\textwidth}
\begin{tikzpicture}[scale=.7]

\draw (0, 0) -- (5, 0);
\draw (0, 1) -- (4, 1);
\draw (1, 2) -- (3, 2);
\draw (4, 2) -- (5, 2);
\draw(0, 3) -- (4, 3);
\draw(0, 4) -- (5, 4);

\draw (0, 0) -- (0, 4);
\draw (0, 1) -- (0, 3);
\draw (1, 1) -- (1, 3);
\draw (2, 0) -- (2, 1);
\draw (2, 3) -- (2, 4);
\draw (3, 1) -- (3, 3);
\draw (4, 0) -- (4, 4);
\draw (5, 0) -- (5, 4);

\end{tikzpicture}
\end{subfigure}
\caption{\label{fig:planar}A sample matching of $R(4, 5)$ and corresponding domino tiling of its planar dual.}
\end{figure}
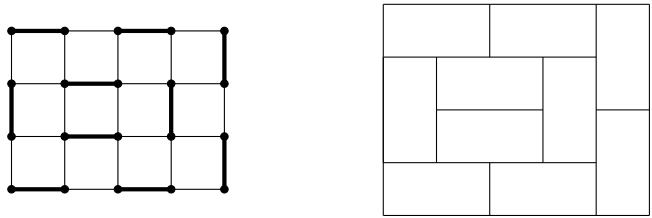

Propp notes that this divisibility property can be proved quite easily from Kasteleyn's original formula expressing the number of matchings of an $m \times n$ rectangle as
$$\#R(m, n) = \displaystyle\prod_{j = 1}^m \prod_{k = 1}^n \left(4 \cos^2 \frac{\pi j}{m + 1} + 4\cos^2 \frac{\pi k}{n + 1}\right)^{1/4}$$
and a dose of abstract algebra. This proof, however, is unsatisfying because it provides no combinatorial insight. A different approach is to express the number of matchings, for fixed $a$, as a linearly recurrent sequence; proving that this particular sequence is a divisibility sequence is possible algebraically for small values of $a$. For example, the sequence $R(2, n)$ is known to be the Fibonnaci sequence shifted by one, which indeed satisfies the divisibility property. This type of proof, too, leads to no combinatorial insight. Both approaches are also very specifically tailored to the case of rectangular regions.

One question is then whether there exists a combinatorial proof of this divisibility property. In this paper, we prove a generalized divisibility theorem, introducing a class of graphs called \emph{compound graphs}. When specialized to rectangles, this theorem provides a proof of the divisibility property of rectangle domino tilings which, despite making use of linear-algebraic techniques, provides combinatorial insight. The key combinatorial lemma of the proof is the ``zero-sum lemma." The zero-sum lemma states that the sum of the signed number of matchings over a family of closely related compound graphs is zero. 

In fact, this zero-sum lemma is related to a couple of other results in the domino tiling literature, such as Ciucu's factorization theorem \cite{ciucu97}. Ciucu's factorization theorem provides a neat proof of the case $\#R(a,b) \mid \#R(a, 2b + 1),$ and thus our divisibility result may be viewed as generalizing Ciucu's theorem. The similarity between these two theorems is not a coincidence; in the last section of the paper, the same linear-algebraic techniques used to prove the zero-sum lemma are applied to provide an alternative proof of Ciucu's lemma when the graph is bipartite. Another related result is Kuo's graphical condensation theorem, or more precisely, his series of graphical condensation relations presented in Section 2 of \cite{kuo03}. Yan and Zhang show that Kuo's graphical condensation relations can be deduced as a special case of Ciucu's lemma \cite{yan05}. 

\begin{figure}[htb]
\centering
\begin{tikzpicture}[scale=.55]
\begin{scope}
\clip[draw] (0, 0) -- (0, 1) -- (1, 1) -- (2, 1) -- (3, 1) -- (3, 2) -- (4, 2) -- (5, 2) -- (6, 2) -- (6, 3) -- (7, 3) -- (8, 3) -- (9, 3) -- (9, 4) -- (10, 4) -- (11, 4) -- (12, 4) -- (12, 5) -- (13, 5) -- (13, 4) -- (14, 4) -- (14, 3) -- (15, 3) -- (15, 2) -- (16, 2) -- (16, 1) -- (17, 1) -- (17, 0) -- (16, 0) -- (15, 0) -- (14, 0) -- (14, -1) -- (13, -1) -- (12, -1) -- (11, -1) -- (11, -2) -- (10, -2) -- (9, -2) -- (8, -2) -- (8, -3) -- (7, -3) -- (6, -3) -- (5, -3) -- (5, -4) -- (4, -4) -- (4, -3) -- (3, -3) -- (3, -2) -- (2, -2) -- (2, -1) -- (1, -1) -- (1, 0) -- cycle;
\draw[fill=gray] (0, -4) grid (17, 5);

\foreach \x in {0,...,17}
\foreach \y in {-4,...,5}
\draw[fill=black] (\x, \y) circle (2pt);

\foreach \x in {0,2,...,16}
\foreach \y in {-4,-2,...,4}
\draw[fill=white] (\x, \y) circle (2pt);

\foreach\x in {1,3,...,17}
\foreach \y in {-3,-1,...,5}
\draw[fill=white] (\x, \y) circle (2pt);

\end{scope}

\foreach \p in {(0, 0), (1, 1), (3, 1), (4, 2), (6, 2), (7, 3), (9, 3), (10, 4), (12, 4), (13, 5), (14, 4), (15, 3), (16, 2), (17, 1), (16, 0), (14, 0), (13, -1), (11, -1), (10, -2), (8, -2), (7, -3), (5, -3), (4, -4), (3, -3), (2, -2), (1, -1)}
\draw[fill=white] \p circle (2pt);

\foreach \p in {(0, 1), (2, 1), (3, 2), (5, 2), (6, 3), (8, 3), (9, 4), (11, 4), (12, 5), (13, 4), (14, 3), (15, 2), (16, 1), (17, 0), (15, 0), (14, -1), (12, -1), (11, -2), (9, -2), (8, -3), (6, -3), (5, -4), (4, -3), (3, -2), (2, -1), (1, 0)}
\draw[fill=black] \p circle (2pt);

\end{tikzpicture}
\caption{\label{fig:aztec}The Aztec pillow graph of order 9}
\end{figure}
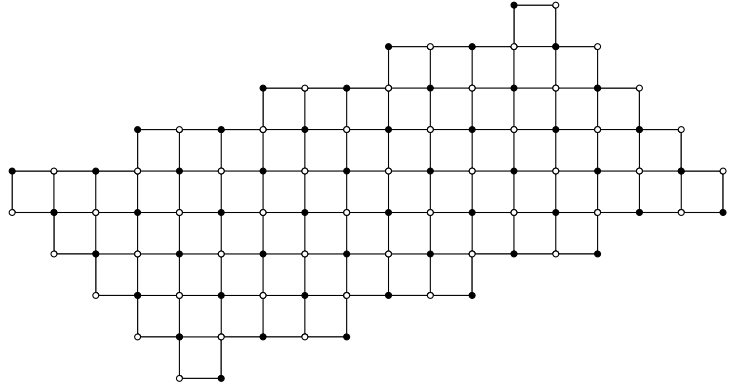

An open problem which we propose to the reader is that of the seemingly analogous divisibility property of Aztec pillows, for which no proof, combinatorial or otherwise, is known to the author. An Aztec pillow is pictured in Figure~\ref{fig:aztec} and was originally presented in \cite{propp99}. If we let $AP_n$ denote the Aztec 3-pillow of width $2n$ (either 0 or 2 mod 4 as the case may be), then it appears that $\#AP_m \mid \#AP_n$ whenever $m + 3 \mid n + 3$. The author has verified this fact computationally for $m,n < 77$. In fact, when $AP_n$ is decomposed into a linearly recursive part and a square part according to Propp's conjecture, \emph{both} parts seem to be divisibility sequences. Strangely enough, neither 5-pillows, 7-pillows, nor 9-pillows (as defined in \cite{hanusa05}) seem to share this divisibility property, although Aztec diamonds do trivially. This mysterious Aztec pillow divisibility, if true, may be of a fundamentally different character from the rectangle divisibility which is the subject of this paper.

\end{section}

\begin{section}{Definition of Compound Graphs}

We provide a constructive definition of compound graphs. Begin by choosing two planar bipartite graphs, the \emph{base graph} and the \emph{supergraph}, and assign to their vertices a bipartite coloring of black and white. We additionally require that the base graph have an equal number of black and white vertices. Place a copy of the base graph, retaining the original coloring, at each vertex of the supergraph. The vertices of the supergraph are thus abstract and represent copies of the base graph. The planar embedding chosen will become important later, so we specify that for each black vertex of the supergraph, the original embedding of the base graph is retained for that copy, while for each white vertex of the supergraph, the embedding is reversed through reflection. In addition, the copies of the base graph must all be placed \emph{outside} of each other, so that all the original faces (except the outer face) are preserved. However, when we refer to the ``outer face" of a copy of the base graph, we will mean the set of edges forming its original outer face (which may no longer be a face after the addition of more edges). An example compound graph is shown in Figure~\ref{fig:ex}.

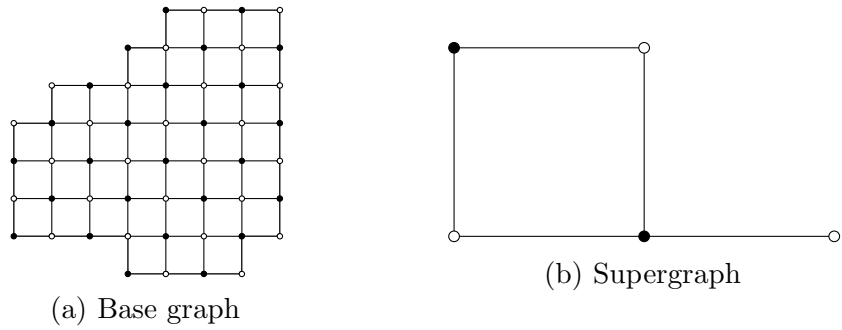
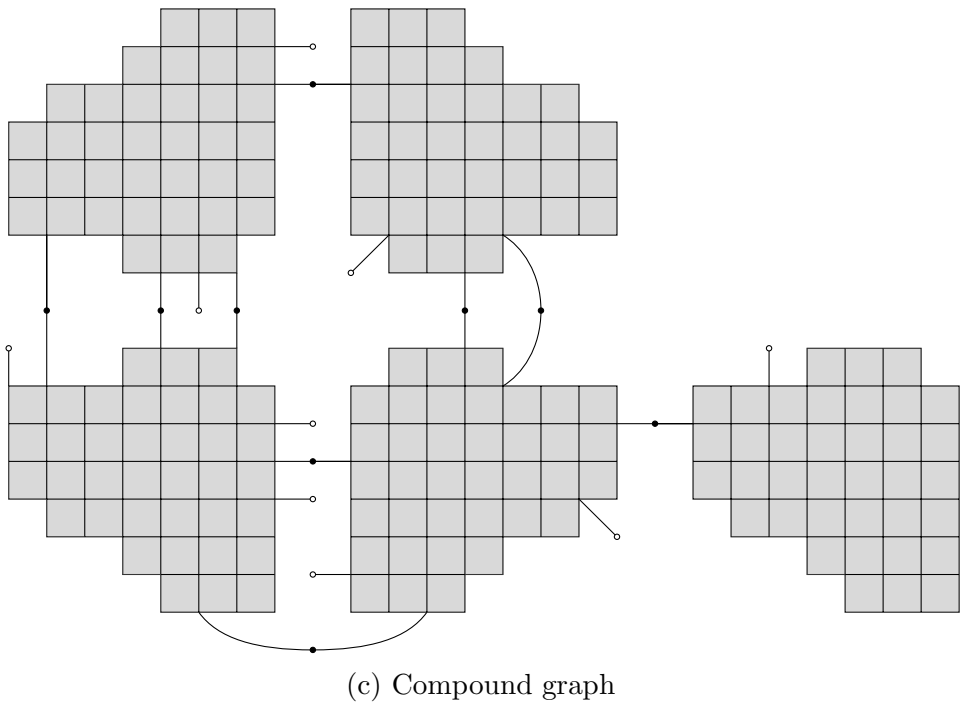
\begin{figure}[ht!]
\centering
\begin{subfigure}{0.4\textwidth}
\centering
\begin{tikzpicture}[scale=.5]

\begin{scope}

\clip[draw] (-2, 1) -- (-2, 2) -- (-1, 2) -- (-1, 3) -- (-1, 4) -- (-1, 5) -- (-1, 6) -- (-1, 7) -- (-1, 8) -- (-2, 8) -- (-3, 8) -- (-4, 8) -- (-4, 7) -- (-5, 7) -- (-5, 6) -- (-6, 6) -- (-7, 6) -- (-7, 5) -- (-8, 5) -- (-8, 4) -- (-8, 3) -- (-8, 2) -- (-7, 2) -- (-6, 2) -- (-5, 2) -- (-5, 1) -- (-4, 1) -- (-3, 1) -- cycle;
\draw (-10, 0) grid (0, 10);
\foreach \x in {0,...,10}
\foreach \y in {0,...,10}
{
	\pgfmathsetmacro{\z}{\x+\y};
	\ifthenelse{\isodd{\z}}
	{\draw[fill=white, xscale=-1] (\x, \y) circle (2pt);}
	{\draw[fill=black, xscale=-1] (\x, \y) circle (2pt);}
}
\end{scope}

\foreach \i in {(2, 1), (1, 2), (1, 4), (1, 6), (1, 8), (3, 8), (4, 7), (5, 6), (7, 6), (8, 5), (8, 3), (7, 2), (5, 2), (4, 1)}
\draw[fill=white, xscale=-1] \i circle (2pt);
\foreach \i in {(2, 2), (1, 3), (1, 5), (1, 7), (2, 8), (4, 8), (5, 7), (6, 6), (7, 5), (8, 4), (8, 2), (6, 2), (5, 1), (3, 1)}
\draw[fill=black, xscale=-1] \i circle (2pt);

\end{tikzpicture}
\caption{Base graph}
\end{subfigure}
\vspace{.5in}
\begin{subfigure}{0.4\textwidth}
\centering
\begin{tikzpicture}

\def\k{2.5};

\draw (0, 0) -- (0, \k) -- (\k, \k) -- (\k, 0) -- cycle;
\draw (\k, 0) -- (\k+\k, 0);

\foreach \i in {(0, \k), (\k, 0)}
\draw[fill=black] \i circle (2pt);

\foreach \i in {(0, 0), (\k, \k), (\k+\k, 0)}
\draw[fill=white] \i circle (2pt);

\end{tikzpicture}
\caption{Supergraph}
\end{subfigure}

\begin{subfigure}{1.0\textwidth}
\centering
\begin{tikzpicture}[scale=.5]

\def\white{(2, 1), (1, 2), (1, 4), (1, 6), (1, 8), (3, 8), (4, 7), (5, 6), (7, 6), (8, 5), (8, 3), (7, 2), (5, 2), (4, 1)};
\def\black{(2, 2), (1, 3), (1, 5), (1, 7), (2, 8), (4, 8), (5, 7), (6, 6), (7, 5), (8, 4), (8, 2), (6, 2), (5, 1), (3, 1)};

\def\x{(2, 1) -- (2, 2) -- (1, 2) -- (1, 3) -- (1, 4) -- (1, 5) -- (1, 6) -- (1, 7) -- (1, 8) -- (2, 8) -- (3, 8) -- (4, 8) -- (4, 7) -- (5, 7) -- (5, 6) -- (6, 6) -- (7, 6) -- (7, 5) -- (8, 5) -- (8, 4) -- (8, 3) -- (8, 2) -- (7, 2) -- (6, 2) -- (5, 2) -- (5, 1) -- (4, 1) -- cycle};

\def\a{15};

\begin{scope}
\clip[draw] \x;
\draw[fill=black!\a] (0, 0) rectangle (10, 10);
\draw (0, 0) grid (10, 10);
\end{scope}

\begin{scope}
\clip[draw, xscale = -1] \x;
\draw[fill=black!\a, xscale=-1] (0, 0) rectangle (10, 10);
\draw[xscale = -1] (0, 0) grid (10, 10);
\end{scope}

\begin{scope}
\clip[draw, yscale=-1] \x;
\draw[fill=black!\a, yscale=-1] (0, 0) rectangle (10, 10);
\draw[yscale = -1] (0, 0) grid (10, 10);
\end{scope}

\begin{scope}
\clip[draw, xscale = -1, yscale = -1] \x;
\draw[fill=black!\a, xscale=-1, yscale=-1] (0, 0) rectangle (10, 10);
\draw[xscale = -1, yscale = -1] (0, 0) grid (10, 10);
\end{scope}

\begin{scope}
\clip[draw, xscale = -1, yscale = -1, shift = {(-18, 0)}] \x;
\draw[fill=black!\a, xscale=-1, yscale=-1, shift = {(-18, 0)}] (0, 0) rectangle (10, 10);
\draw[draw, xscale = -1, yscale = -1, shift = {(-18, 0)}] (0, 0) grid (10, 10);
\end{scope}

\coordinate (s0) at (-2, 0);
\coordinate (s1) at (-4, 0);
\coordinate (s2) at (4, 0);
\coordinate (s3) at (6, 0);
\coordinate (s4) at (-7, 0);
\coordinate (s5) at (0, -4);
\coordinate (s6) at (0, -9);
\coordinate (s7) at (0, 6);
\coordinate (s8) at (9, -3);

\foreach \i in {0,1,2}
{
	\draw (s\i) -- +(0, 1);
	\draw (s\i) -- +(0, -1);
	\draw[fill=black] (s\i) circle (2pt);
}

\draw (s3) .. controls (6, 1) and (5.5, 1.7) .. (5, 2);
\draw[yscale = -1] (s3) .. controls (6, 1) and (5.5, 1.7) .. (5, 2);
\draw[fill=black] (s3) circle (2pt);

\draw (s4) -- +(0, 2) -- +(0, -2);
\draw[fill=black] (s4) circle (2pt);

\draw (s5) -- +(1, 0) -- +(-1, 0);
\draw[fill=black] (s5) circle (2pt);

\draw (s6) .. controls (1.5, -9) and (2.5, -8.7) .. (3, -8);
\draw[xscale=-1] (s6) .. controls (1.5, -9) and (2.5, -8.7) .. (3, -8);
\draw[fill=black] (s6) circle (2pt);

\draw (s7) -- +(1, 0) -- +(-1, 0);
\draw[fill=black] (s7) circle (2pt);

\draw (s8) -- +(1, 0) -- +(-1, 0);
\draw[fill=black] (s8) circle (2pt);

\draw[xscale=-1] (-1, 1) -- (-2, 2);
\draw[fill=white, xscale=-1] (-1, 1) circle (2pt);

\draw (-1, -3) -- (0, -3);
\draw[fill=white] (0, -3) circle (2pt);

\draw (-1, -5) -- (0, -5);
\draw[fill=white] (0, -5) circle (2pt);

\draw (0, -7) -- (1, -7);
\draw[fill=white] (0, -7) circle (2pt);

\draw (-3, 0) -- (-3, 1);
\draw[fill=white] (-3, 0) circle (2pt);

\draw (8, -6) -- (7, -5);
\draw[fill=white] (8, -6) circle (2pt);

\draw (0, 7) -- (-1, 7);
\draw[fill=white] (0, 7) circle (2pt);

\draw (12, -1) -- (12, -2);
\draw[fill=white] (12, -1) circle (2pt);

\draw (-8, -1) -- (-8, -2);
\draw[fill=white] (-8, -1) circle (2pt);

\end{tikzpicture}
\caption{Compound graph}
\end{subfigure}

\caption{\label{fig:ex}An example compound graph with corresponding base graph and supergraph pictured. Although the base graph shown here is a subgraph of the square lattice, this need not be the case.}
\end{figure}

The vertices belonging to the copies of the base graph are called \emph{base vertices}, the edges \emph{base edges}, and the faces \emph{base faces}. Define an equivalence relation $\cong$ on the base vertices by saying $u \cong v$ if and only if $u$ and $v$ correspond to the same vertex of the base graph. In this sense, each vertex of the original base graph may be thought of as an equivalence class. Analogously, let us say that two bases edges are equivalent if and only if they correspond to the same edge of the base graph. Because no two vertices on different copies of the base graph are or will be adjacent, two edges are equivalent if and only if their respective endpoints are equivalent. 

Next, a number of vertices and edges shall be added to the graph in a certain way, subject to the condition that the resulting graph must remain planar. For each edge of the supergraph, whose endpoints correspond to two copies of the base graph, say, $G_1$ and $G_2$, we can place one or more vertices which we call \emph{stems}. Each stem is then connected to exactly one vertex on the outer face of each of $G_1$ and $G_2$, and the two vertices must be equivalent. We will furthermore specify that all the stems are colored black, so that they must connect to white vertices. 

Suppose $p$ stems have been added in this way. Then, the final modification is to place $p$ vertices referred to as \emph{leaves}, which have degree 1 and connect to any of the black vertices lying on the outer face of a copy of the base graph. The leaves will all be white. (When we refer to leaves in the following discussion, we refer to these leaves, not to any base vertices which happen to have degree 1.) 

Any graph thus obtained by this procedure involving a base graph, supergraph, and stems and leaves is called a compound graph. The fact that the stems are black and the leaves are white is important only for the proof of the zero-sum lemma; in the last section, we will remove that constraint, but require that the supergraph have only two vertices.


The subsequent results will be stated for weighted graphs (and may be specialized to the case of unweighted graphs by setting all the weights equal to 1). The weight of a matching is defined to be the product of the weights of all the edges in the matching. The ``number of matchings" of a weighted graph $G$, written symbolically as $\#G$, refers to the sum of the weights of all matchings of $G$. 

Given a weighted base graph, the weights of a compound graph are assigned as follows. Each base edge of the compound graph is assigned the weight of the edge in $E$ to which it corresponds. The weights of the edges attached to stems or leaves may be assigned arbitrarily, with the condition that any two edges attached to the same stem are given the same weight. The effect of these weights is merely to multiply the number of matchings by a factor. For simplicity, we will assume that each of these edges has a weight of 1.

The weights may lie in any integral domain $R$. Note that, although setting $R$ to be the integers is perhaps the most natural choice, $R$ can be much more general. For example, we can let each base edge define a different variable as its weight, and let $R$ be the ring of integer polynomials in these variables.

The main result, which will be proved in Section ~\ref{sec:div}, is that the weighted number of matchings of a compound graph is divisible by the weighted number of matchings of its base graph.

\end{section}

\begin{section}{Kasteleyn's Theorem and Sign Functions}

Kasteleyn's theorem expresses the number of matchings of a graph as the determinant of a certain matrix, with the aid of a \emph{sign function} defined on the edges. This technique has been described in various forms in the literature; for example, Kasteleyn originally presents the technique in terms of choosing a ``Pfaffian orientation" for a planar graph which need not be bipartite \cite{kasteleyn67}. For the reader's convenience, we restate Kasteleyn's theorem in the bipartite form used by Percus \cite{percus69} and present the necessary definitions and propositions involving sign functions. The first two facts, Lemma~\ref{lem:faces+} and Lemma~\ref{lem:sgn exists}, are known and draw from Kasteleyn's original papers \cite{kasteleyn61, kasteleyn67} and Kenyon's expository notes \cite{kenyon09}. At the end of this section, we present a new lemma regarding sign functions on compound graphs.

Suppose we have a planar bipartite graph with edge set $E$. A sign function is a function $\s\colon E \rightarrow \{1, -1\}$ such that if $\mathcal C$ is a simple cycle of length $2\ell$, strictly enclosing an even number of points, then 
$$\displaystyle\prod_{e \in \C} \s e = (-1)^{\ell - 1}.$$

Unfortunately, sign functions are not independent of which planar embedding of a graph is chosen. Therefore, when we discuss sign functions on compound graphs, we will have to specify a planar embedding, up to homeomorphism. We have already described the embedding of the copies of the base graph. In the subsequent discussion, let us specify that the leaves shall always be drawn on the outside of the copies of the base graph, so that the original faces of the base graph are preserved. Any planar embedding of the stems may be chosen.

Now, given a sign function, Kasteleyn's theorem allows us to calculate the number of perfect matchings. In the original formulation of the theorem, a matrix is used, requiring an arbitrary enumeration of the vertices. To avoid the messiness of indexing vertices, we recast Kasteleyn's theorem using a linear operator, and view it when convenient as a matrix whose rows and columns are indexed by vertices, modulo permutation of the vertices.

Let $\G$ be the set of planar bipartite graphs, with vertices colored black and white, such that there are an equal number of black and white vertices. (If the number of black and white vertices differs, then there are clearly no matchings.) Note that, by definition, both a compound graph and its base graph are members of $\G$. 

For each graph $H$ in $\G$, define the following two vector spaces: Let $V_b$ be the space of $F$-valued functions on the black vertices, and let $V_w$ be the space of $F$-valued functions on the white vertices, where $F$ is the quotient field of the ring $R$ containing the edge weights. Addition and scalar multiplication are defined as usual. For each vector space, we take the set of indicator functions for each vertex as the standard basis.

Now we define $K\colon V_w \rightarrow V_b$ as the linear operator represented by the matrix $K = (K_{uv})$ for black vertices $u$ and white vertices $v$, with
$$K_{uv} = \begin{cases} \s(u, v) \w(u, v) &\mbox{if } u \sim v \\ 0 &\mbox{otherwise.} \end{cases}$$
(We write $u \sim v$ if $u$ and $v$ are connected by an edge, and let $(u, v)$ denote the corresponding edge.) So in other words, for any function $f \in V_w$, we have $$(Kf)(u) = \sum_{v \sim u} f(v) \s(u, v) \w(u, v).$$
Since an ordering of the vertices has not been chosen, the determinant of $K$ is defined only up to sign. 

\begin{theorem}[The Kasteleyn-Percus Theorem]
The determinant of $K$ equals, up to sign, the number of matchings of $H$.
\end{theorem}

\begin{proof}
The proof is well-known, and can be found in \cite{percus69}.
\end{proof}

We return to a discussion of sign functions, with the primary goal of proving their existence. In order to describe sign functions on compound graphs, we must make use of several propositions.

First, however, let us clarify some terminology. By ``faces" of a graph, we will be referring only to inner faces, excluding the outer face, unless otherwise specified. By ``edges" of a face we will refer to the multiset of edges encountered when traversing the boundary of the face. Thus, in Figure~\ref{fig:multiplicities}, the edges consist of $e_0, e_1, \hdots, e_{11}$ even though $e_1 = e_2$, $e_8 = e_{11}$ and $e_9 = e_{10}$. However, these pairs of equal edges will not be regarded as distinct in the edge set $E$. 

Finally, let $F$ be a face of the graph with $2m$ edges. Given any function $f$ from the edges to $\{1, -1\}$, let us call $F$ \emph{positive} with respect to $f$ if the product of $f$ over all edges of $F$ is $(-1)^{m - 1}$, and \emph{negative} otherwise. 

\begin{figure}[htb]
\centering
\begin{tikzpicture}[scale=.8]

\coordinate (p0) at (90:3);
\coordinate (p1) at (30:3);
\coordinate (p2) at (330:3);
\coordinate (p3) at (270:3);
\coordinate (p4) at (210:3);
\coordinate (p5) at (150:3);
\coordinate (p6) at (30:1);
\coordinate (p7) at (280:1.2);
\coordinate (p8) at (150:1.1);

\draw (p0) -- 
node[below] {$e_0$} (p1);
\draw (p1) --
node[above] {$e_1$} 
node[below] {$e_2$} (p6);
\draw (p1) --
node[left] {$e_3$} (p2);
\draw (p2) --
node[above] {$e_4$} (p3);
\draw (p3) --
node[above] {$e_5$} (p4);
\draw (p4) --
node[right] {$e_6$} (p5);
\draw (p5) --
node[below] {$e_7$} (p0);
\draw (p0) --
node[left] {$e_8$} 
node[right] {$e_{11}$} (p8);
\draw (p8) --
node[left] {$e_9$}
node[right] {$e_{10}$} (p7);

\foreach \x in {0,2,4,6,7}
\draw[fill=black] (p\x) circle (2pt);

\foreach \x in {1,3,5,8}
\draw[fill=white] (p\x) circle (2pt);

\end{tikzpicture}
\caption{\label{fig:multiplicities}Counting multiplicities of edges of a face}
\end{figure}
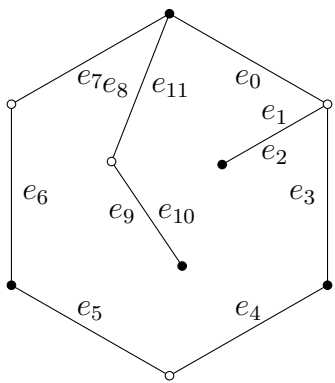

\begin{lemma}\label{lem:faces+}
If all faces in $G$ are positive with respect to $f$, then $f$ is a sign function.
\end{lemma}

\begin{proof}
We will prove the stronger statement that if all faces in $G$ are positive, then for all simple cycles $\C$ of length $|\C|$, strictly enclosing $A(\C)$ points,
$$\displaystyle\prod_{e \in \C} \s f = (-1)^{|\C|/2 + A(\C) - 1}.$$
The proof will be by induction on the number of faces contained in $\C$.

For the base case, suppose $\C$ contains only one face. The edges of the face form a cycle $\C^\prime,$ which may not necessarily be simple. If $\C^\prime = \C$, then the claim follows trivially. If not, then $\C^\prime$ consists of $\C$, together with one or more doubled trees extending into the interior of the face. Suppose that $A(\C)$ vertices are strictly enclosed by $\C.$ Then it is clear that the sum of the lengths of the trees must be $A(\C)$, so that $|\C^\prime| = |\C| + 2A(\C)$, where $|\cdot|$ denotes the length of the cycle, including multiplicities. Since the trees are all doubled,
\begin{align*}
\displaystyle\prod_{e \in \C} \s f &= \prod_{e \in \C^\prime} \s f\\
&= (-1)^{|\C^\prime|/2 - 1} \\
&= (-1)^{|\C|/2 + A(\C) - 1},
\end{align*}
as desired.

Now suppose $\C$ contains more than one face. Then the faces of $\C$ can be decomposed into the disjoint union of the faces contained by two smaller simple cycles, $\C_1$ and $\C_2$. Suppose the intersection of $\C_1$ and $\C_2$ is a path with $\ell$ vertices. We have, by the inductive hypothesis,
\begin{align*}
\displaystyle\prod_{e \in \C} \s f &= \prod_{e \in \C_1} \s f \prod_{e \in \C_2} \s f \\
&= (-1)^{|\C_1|/2 + A(\C_1) - 1} (-1)^{|\C_2|/2 + A(\C_2) - 1}.
\end{align*}
Since $|\C_1| + |\C_2| = |\C| + 2\ell - 2$ and $A(\C) = A(\C_1) + A(\C_2) + \ell - 2,$ we have
$$\displaystyle\prod_{e \in \C} \s f  = (-1)^{|\C|/2 + A(\C) - 1},$$
completing the proof.
\end{proof}

\begin{lemma}\label{lem:sgn exists}
Every planar bipartite graph has a sign function.
\end{lemma}

\begin{proof}
The main idea is the simple observation that negating the value of $f$ on a particular edge of face $F$ switches $F$ from positive to negative, or vice versa. We will prove the statement by induction on the number of faces of the graph. 

In the base case, when the graph has no faces, there exists a sign function vacuously. Now assume that there exists a sign function for every graph with $k - 1$ faces, and consider a graph with $k$ faces. Choose an edge $e_0$ which borders the outer face, and suppose the inner face it borders is face $F$. Now if we delete all the edges which belong to $F$ but no other face of the graph, we are left with a graph $G^\prime$ with $k - 1$ faces. By the inductive hypothesis, there exists a sign function $f^\prime$ on this graph. For the original graph $G$, we now define $f(e) = f^\prime(e)$ if $e$ belongs to $G^\prime$, and $f(e) = 1$ otherwise. By definition, every face other than $F$ must be positive. If $F$ is positive, we are done. If not, set $f(e_0) = -1$. Then $F$ becomes positive while leaving all the other faces unchanged, so $f$ is a sign function.
\end{proof}

Now we would like to describe sign functions $\s$ on an arbitrary compound graph $H$ such that $\s e_1 = \s e_2$ whenever $e_1 \cong e_2$ (we call this ``respecting the equivalence relation"). Call the set of such functions $S(H)$. The ultimate goal of our discussion is to show that $S(H)$ is non-empty.

First, however, we must introduce one more related graph, the \emph{reduced graph}. The reduced graph, denoted $R$, is formed by taking the compound graph, for each copy of the base graph identifying all its vertices, then deleting the resulting self-loops. The planar embedding of the stems and leaves are retained. The reduced graph is also a planar bipartite graph, where the copies of the base graph are of one color and the stems and leaves are of another. In general, $R$ has no perfect matchings, but it will still be convenient to consider sign functions on $R$. A reduced graph with an accompanying sign function is shown in Figure~\ref{fig:reduced}.

\begin{figure}[htb]
\centering
\begin{tikzpicture}

\draw (0, 4) -- (2, 4) -- (4, 4);
\draw (4, 0) -- (4, 2) -- (4, 4);
\draw (0, 4) .. controls (.3, 4.1) and (1, 4.5) .. (2, 4.5);
\draw (4, 0) .. controls (4.1, .3) and (4.5, 1) .. (4.5, 2);
\draw[style=dashed] (4.5, 2) .. controls (4.5, 3) and (4.1, 3.7) .. (4, 4);
\draw (3.5, 2) .. controls (3.5, 3) and (3.9, 3.7) .. (4, 4);
\draw (0, 0) -- (2, 0);
\draw (0, 0) .. controls (.3, .1) and (1, .5) .. (2, .5) .. controls (3, .5) and (3.7, .1) .. (4, 0);
\draw (2, -.5) .. controls (3, -.5) and (3.7, -.1) .. (4, 0);
\draw (0, 0) .. controls (.3, .15) and (.9, 1) .. (2, 1);
\draw[style=dashed] (0, 0) -- (0, 2);
\draw (0, 2) -- (0, 4);
\draw[style=dashed] (0, 0) .. controls (.15, .3) and (1, .9) .. (1, 2);
\draw (1, 2) .. controls (1, 3.1) and (.15, 3.7) .. (0, 4);
\draw (.5, 2) .. controls (.5, 3) and (.1, 3.7) .. (0, 4);
\draw (0, 0) .. controls (-.1, .3) and (-.5, 1) .. (-.5, 2) .. controls (-.5, 3) and (-.1, 3.7) .. (0, 4);
\draw (-1, 2) .. controls (-1, 1) and (-.4, .3) .. (0, 0);
\draw (0, 0) .. controls (.3, -.4) and (1, -1) .. (2, -1);
\draw[dashed] (2, -1) .. controls (3, -1) and (3.7, -.4) .. (4, 0);
\draw (4, 0) -- (6, 0) -- (8, 0) -- (8, 2);
\draw (4, 0) .. controls (4.3, -.1) and (5, -.5) .. (6, -.5);

\foreach \i in {(0, 0), (4, 0), (0, 4), (4, 4), (8, 0), (1, 2), (2, 1), (2, 4), (2, 4.5), (4, 2), (4.5, 2), (3.5, 2), (2, 0), (2, .5), (2, -.5), (0, 2), (.5, 2), (-.5, 2), (-1, 2), (2, -1), (6, 0), (8, 2), (6, -.5)}
\draw[fill=black] \i circle (2pt);

\end{tikzpicture}
\caption{\label{fig:reduced}The reduced graph of the compound graph in Figure~\ref{fig:ex}, labeled with a sign function. Solid lines represent a weight of 1, while dashed edges represent a weight of $-1$.}
\end{figure}
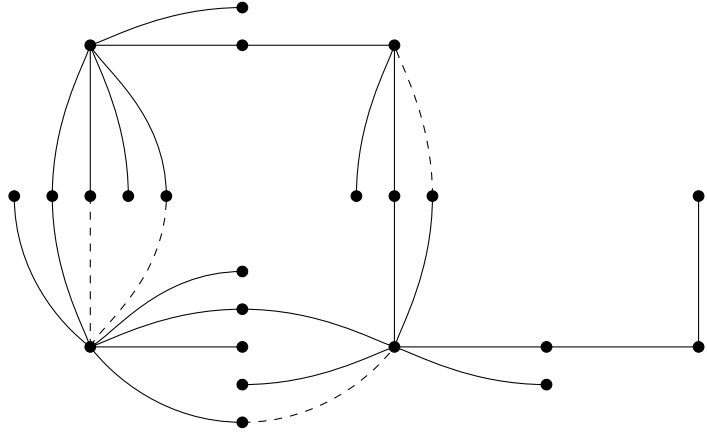

Let $S(G)$ denote the set of sign functions on the base graph of $H$, and let $S(R)$ denote the set of sign functions on the reduced graph of $H$.

\begin{lemma}\label{lem:restriction}
There exists a one-to-one correspondence $\phi\colon S(H) \rightarrow S(G) \times S(R)$.  
\end{lemma}

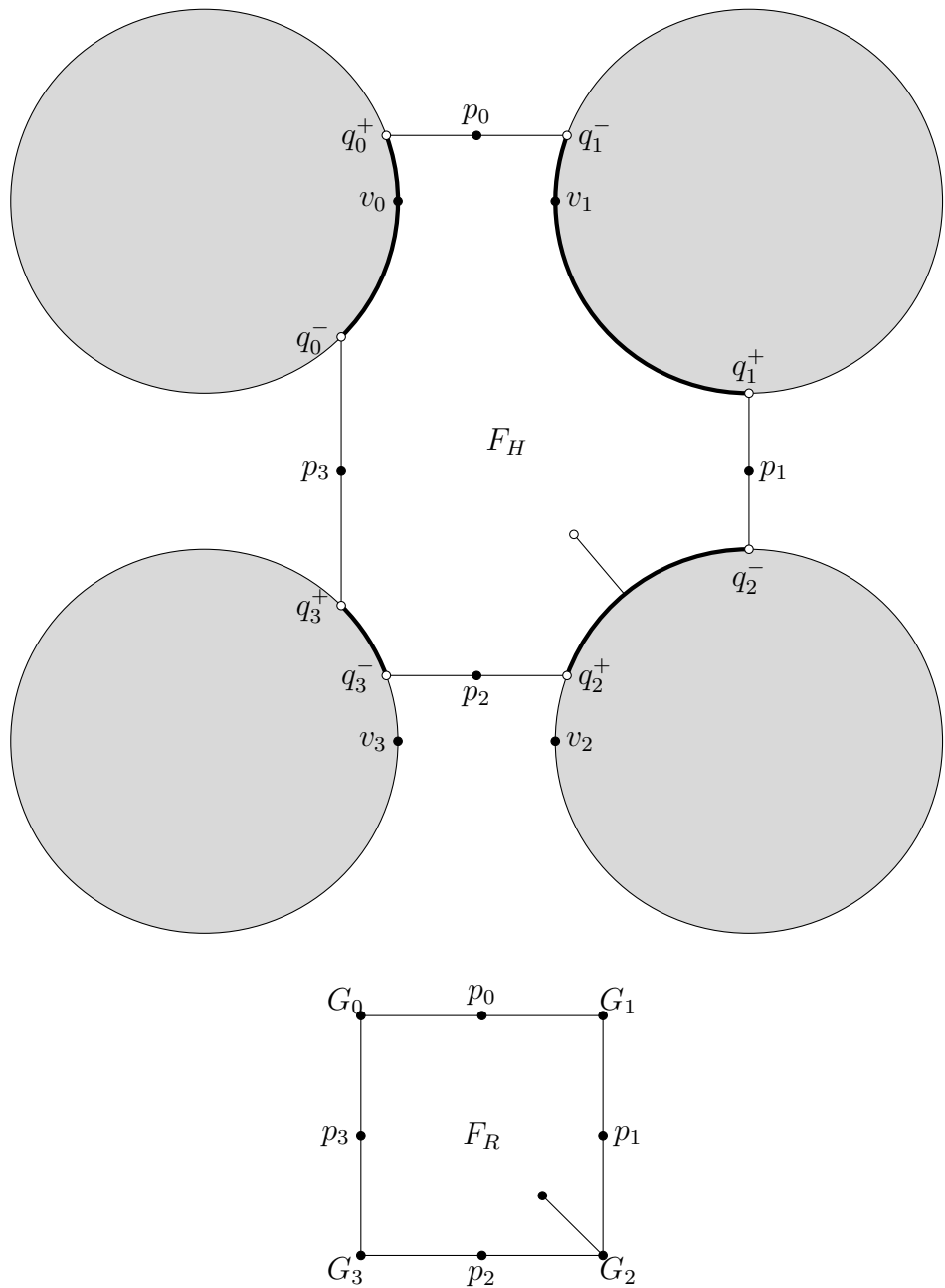
\begin{figure}[ht!]
\begin{subfigure}{1.0\textwidth}
\centering
\begin{tikzpicture}[scale=.8]

\def\s{9};

\coordinate (p0) at (0, \s);
\coordinate (p1) at (\s, \s);
\coordinate (p2) at (\s, 0);
\coordinate (p3) at (0, 0);

\def\r{3.2};
\def\a{20};
\def\b{90};
\def\c{20};
\def\d{45};

\coordinate (q0+) at ($(p0) + (\a:\r)$);
\coordinate (q1-) at ($(p1) + (180 - \a:\r)$);
\coordinate (q1+) at ($(p1) + (-\b:\r)$);
\coordinate (q2-) at ($(p2) + (\b:\r)$);
\coordinate (q2+) at ($(p2) + (180 - \c:\r)$);
\coordinate (q3-) at ($(p3) + (\c:\r)$);
\coordinate (q3+) at ($(p3) + (\d:\r)$);
\coordinate (q0-) at ($(p0) + (-\d:\r)$);

\foreach \x in {0,...,3}
{
	\draw (p\x) node {$G_\x$};
	\draw[fill=black!15] (p\x) circle (\r cm);
}

\draw (5, 5) node {$F_H$};

\draw[fill=black] (p0) + (\r,0) 
node[left] {$v_0$}
circle (2pt);
\draw[fill=black] (p1) + (-\r, 0)
node[right] {$v_1$}
circle (2pt);
\draw[fill=black] (p2) + (-\r, 0)
node[right] {$v_2$}
circle (2pt);
\draw[fill=black] (p3) + (\r, 0)
node[left] {$v_3$}
circle (2pt);

\draw (q0+) node[left] {$q_0^+$};
\draw (q1-) node[right] {$q_1^-$};
\draw (q1+) node[above] {$q_1^+$};
\draw (q2-) node[below] {$q_2^-$};
\draw (q2+) node[right] {$q_2^+$};
\draw (q3-) node[left] {$q_3^-$};
\draw (q3+) node[left] {$q_3^+$};
\draw (q0-) node[left] {$q_0^-$};

\draw[ultra thick] (q0+) arc (\a:-\d:\r);
\draw[ultra thick] (q1+) arc (360-\b:180 - \a:\r);
\draw[ultra thick] (q2+) arc (180 - \c:\b:\r);
\draw[ultra thick] (q3+) arc (\d:\c:\r);

\draw (q0+) -- (q1-);
\draw (q1+) -- (q2-);
\draw (q2+) -- (q3-);
\draw (q3+) -- (q0-);

\draw[fill=black] ($(q0+)!.5!(q1-)$)
node[above] {$p_0$}
circle (2pt);
\draw[fill=black] ($(q1+)!.5!(q2-)$)
node[right] {$p_1$}
circle (2pt);
\draw[fill=black] ($(q2+)!.5!(q3-)$)
node[below] {$p_2$}
circle (2pt);
\draw[fill=black] ($(q3+)!.5!(q0-)$)
node[left] {$p_3$}
circle (2pt);

\foreach \x in {0,...,3}
\draw[fill=white] (q\x-) circle (2pt) (q\x+) circle (2pt);

\def\theta{130};

\draw (p2) + (\theta:4.5) -- ($(p2) + (\theta:\r)$);
\draw[fill=white] (p2) + (\theta:4.5) circle (2pt);

\end{tikzpicture}
\vspace{.2in}
\end{subfigure}
\begin{subfigure}{1.0\textwidth}
\centering
\begin{tikzpicture}[scale=.8]

\coordinate (p0) at (2, 4);
\coordinate (p1) at (4, 2);
\coordinate (p2) at (2, 0);
\coordinate (p3) at (0, 2);
\coordinate (G0) at (0, 4);
\coordinate (G1) at (4, 4);
\coordinate (G2) at (4, 0);
\coordinate (G3) at (0, 0);

\def\pt{.25};

\foreach \x in {0, 1, 2, 3}
\draw[fill=black] (p\x) circle (2pt) (G\x) circle (2pt);

\draw (G0) -- (p0) -- (G1) -- (p1) -- (G2) -- (p2) -- (G3) -- (p3) -- cycle;
\draw (G2) -- (3, 1);

\draw (p0) node[above] {$p_0$};
\draw (p1) node[right] {$p_1$};
\draw (p2) node[below] {$p_2$};
\draw (p3) node[left] {$p_3$};

\draw (2, 2) node {$F_R$};

\draw (G0) + (-\pt, \pt) node {$G_0$};
\draw (G1) + (\pt, \pt) node {$G_1$};
\draw (G2) + (\pt, -\pt) node {$G_2$};
\draw (G3) + (-\pt, -\pt) node {$G_3$};

\draw[fill=black] (3, 1) circle (2pt);

\end{tikzpicture}
\end{subfigure}
\caption{\label{fig:reduced-positivity}Proof that $F_H$ is positive if and only if $F_R$ is positive.}
\end{figure}

\begin{proof}
First, let $\s|_G$ denote taking the restriction of $\s$ to $G$. This map is well-defined because of the condition on sign functions in $S(H)$. Likewise, define $\s|_R$ to be the restriction of $\s$ to $R$. We claim that restriction to $G$ and $R$ determines a one-to-one correspondence between $S(H)$ and $S(G) \times S(R)$. 

Consider a sign function $\s$ on $H$ and its restrictions $\s|_G$ and $\s|_R$. We will say faces of $H$ are positive or negative with respect to $\s$, faces of $G$ are positive or negative with respect to $\s|_G$, and faces of $R$ are positive or negative with respect to $\s|_R$. Note that the faces of $H$ may be partitioned into two categories: the set $A$ of faces which correspond to faces of $G$, and the set $B$ of faces which correspond to faces of $R$. By definition, all faces in $A$ are positive if and only if all faces of $G$ are positive. We must show that all faces in $B$ are positive if and only if all faces in $R$ are positive.

Choose any face $F_H$ in $B$, and suppose that when enumerated in clockwise order the stems read $p_0, p_1, \hdots, p_{\ell - 1},$ where $\ell$ is even since the supergraph is bipartite. Suppose also that each stem $p_i$ connects to copies $G_i$ and $G_{i + 1}$ of the base graph, modulo $\ell$. Let each $p_i$ connect to white vertices $q_i^+$ on $G_i$ and $q_{i + 1}^-$ on $G_{i + 1}$. See Figure~\ref{fig:reduced-positivity} for an illustration. Now the corresponding face $F_R$ of $R$ consists only of the edges incident to the stems and any leaves which happen to lie inside the face. For any vertices $u$ and $v$ on the outer face of a copy of the base graph, let $P(u, v)$ be the product of $\s$ over the edges of the path travelling counterclockwise along the outer face from $u$ to $v$. Thus,
$$\displaystyle\prod_{e \in F_H} \s e = \prod_{e \in F_R} \s e \prod_{i = 0}^{\ell-1} P(q_i^-, q_i^+).$$

Choose a vertex $v$ of the base graph, and let the vertex of copy $G_i$ corresponding to vertex $v$ be denoted $v_i$. Note that if $v_i$ lies on the path counterclockwise from $q_i^-$ to $q_i^+$, then $P(q_i^-, q_i^+) = P(q_i^-, v_i)P(v_i, q_i^+)$. Otherwise, $P(q_i^-, q_i^+) = CP(q_i^-, v_i)P(v_i, q_i^+),$ where $C$ denotes the product of $f$ along the entire outside face of $G$. So
$$\displaystyle\prod_{i = 0}^{\ell-1} P(q_i^-, q_i^+) = C^m \prod P(q_i^-, v_i)P(v_i, q_i^+),$$
where $m$ is the number of copies of the base graph such that $v_i$ lies outside the path from $q_i^-$ to $q_i^+$. 

We claim that $m$ must be even. Imagine Alice is an ant traveling clockwise around $F_H$ from vertex to vertex. Meanwhile, Bob travels along the outer face of the base graph $G$, such that Bob is always at the base vertex corresponding to the vertex Alice is at (ignoring the stems $p_i$). Now as Alice travels around $F_H$, Bob zig-zags back and forth along the outer face, turning around each time after he has visited the vertex corresponding to $q_i^+$ (equivalently, $q_{i + 1}^-$) twice. When Alice completes her circuit, Bob returns to his starting point. So Bob must have visited the base vertex $v$ an even number of times, or he would not have been able to return to his starting point. Consequently, Alice must have visited an even number of the corresponding vertices $v_i$. 

Therefore, $C^m = 1$. Since $q_i^+ \cong q_{i + 1}^-$ and $v_i \cong v_{i+1}$, and they are on neighboring copies of the base graph, we have $P(v_i, q_i^+) = P(q_{i+1}^-, v_{i+1}).$ Thus 
\begin{align*}
\displaystyle\prod_{i = 0}^{\ell-1} P(q_i^-, q_i^+) &= \prod P(q_i^-, v_i)P(v_i, q_i^+) \\
&= \prod P(q_{i + 1}^-, v_{i+1}) P(v_i, q_i^+) \\
&= \prod P(v_i, q_i^+)^2 \\
&= 1,
\end{align*}
so 
$$\prod_{e \in F_H} \s e = \prod_{e \in F_R} \s e.$$
Since all the $q_i$ are white, it follows that the length of each path from $q_i^-$ to $q_i^+$ is even, so the difference between the number of edges in $F_H$ and $F_R$ is divisible by 4. The number of points strictly inside each face is also the same. Therefore $F_H$ is positive if and only if $F_R$ is positive, as desired.
\end{proof}

The importance of this lemma is that it shows, together with Lemma ~\ref{lem:sgn exists}, that there exists a sign function which respects the equivalence relation on any compound graph. It is also useful for application to graphs, by allowing us to easily determine the values of the sign function on the stems and leaves.

Given a compound graph $H$, we now define a \emph{sign-weighted compound graph}. Take a sign function in $S(H)$, and modify the weights on all edges $e$ attached to stems or leaves by multiplying the weight by $\s e$. The sign-weighted version of a graph $H$ will be denoted by $\swh$, where $\swh$ is understood to depend on the choice of sign function. As we will see, the process of sign-weighting compound graphs is necessary so that afterwards, in the Kasteleyn matrix, the entries corresponding to the edges of the reduced graph assume values independent of the topological configuration of the reduced graph.
\end{section}

\begin{section}{The Divisibility Theorem}
\label{sec:div}
We begin by presenting the zero-sum lemma. Consider a compound graph $H$, and pick a leaf $p$ attached to vertex $q$. A sibling of $H$ is a compound graph which can be formed by deleting $p$ and adding a new leaf $p^\prime$ attached to any other vertex $q^\prime \cong q$. The set containing $H$ together with its siblings shall be known as a \emph{family}, with respect to leaf $p$. A family of compound graphs is shown in Figure~\ref{fig:family}. Note that although technically the vector spaces $V_b$ and $V_w$ are different for each member of the family, by identifying the leaf $p$ as it moves around we can treat them as the same vector space, so that the Kasteleyn operator is defined on the same space for each sibling. 

\begin{figure}[htb]
\centering
\begin{subfigure}{1.0\textwidth}
\hspace{1.2in}
\begin{tikzpicture}[scale=.8]

\foreach \x in {3, 7}
\foreach \y in {1, 3}
{
	\draw (\x - 1, \y) -- (\x + 1, \y);
}

\draw (2, 2) -- +(1, 0) node[above] {$p$};
\draw (2, 0) -- +(1, 0);
\draw (6, 2) -- +(1, 0);
\draw (9, 3) -- +(0, 1);

\foreach \x in {0, 4, 8}
{
	\draw[fill=black!15] (\x, 0) rectangle (\x + 2, 3);
	\draw (\x, 0) grid (\x + 2, 3);
	\foreach \i in {0,...,2}
	\foreach \j in {0,...,3}
	{
		\pgfmathsetmacro{\k}{\i+\j};
		\ifthenelse{\isodd{\k}}
		{\draw[fill=white] (\x + \i, \j) circle (2pt);}
		{\draw[fill=black] (\x + \i, \j) circle (2pt);}
	}
}

\foreach \x in {3, 7}
\foreach \y in {1, 3}
{
	\draw[fill=black] (\x, \y) circle (2pt);
}

\draw[fill=white] (3, 2) circle (2pt);
\draw[fill=white] (3, 0) circle (2pt);
\draw[fill=white] (7, 2) circle (2pt);
\draw[fill=white] (9, 4) circle (2pt);

\end{tikzpicture}
\vspace{.3in}
\end{subfigure}
\vspace{.3in}
\begin{subfigure}{1.0\textwidth}
\hspace{1.2in}
\begin{tikzpicture}[scale=.8]
\foreach \x in {3, 7}
\foreach \y in {1, 3}
{
	\draw (\x - 1, \y) -- (\x + 1, \y);
}

\draw (3, 2) node[above] {$p$} -- +(1, 0);
\draw (2, 0) -- +(1, 0);
\draw (6, 2) -- +(1, 0);
\draw (9, 3) -- +(0, 1);

\foreach \x in {0, 4, 8}
{
	\draw[fill=black!15] (\x, 0) rectangle (\x + 2, 3);
	\draw (\x, 0) grid (\x + 2, 3);
	\foreach \i in {0,...,2}
	\foreach \j in {0,...,3}
	{
		\pgfmathsetmacro{\k}{\i+\j};
		\ifthenelse{\isodd{\k}}
		{\draw[fill=white] (\x + \i, \j) circle (2pt);}
		{\draw[fill=black] (\x + \i, \j) circle (2pt);}
	}
}

\foreach \x in {3, 7}
\foreach \y in {1, 3}
{
	\draw[fill=black] (\x, \y) circle (2pt);
}

\draw[fill=white] (3, 2) circle (2pt);
\draw[fill=white] (3, 0) circle (2pt);
\draw[fill=white] (7, 2) circle (2pt);
\draw[fill=white] (9, 4) circle (2pt);
\end{tikzpicture}
\end{subfigure}
\begin{subfigure}{1.0\textwidth}
\hspace{1.2in}
\begin{tikzpicture}[scale=.8]
\foreach \x in {3, 7}
\foreach \y in {1, 3}
{
	\draw (\x - 1, \y) -- (\x + 1, \y);
}
\draw[white] (3, 3) -- (4, 3);
\draw[dashed] (3, 3) -- (4, 3);

\draw (10, 2) -- +(1, 0) node[above] {$p$};
\draw (2, 0) -- +(1, 0);
\draw (6, 2) -- +(1, 0);
\draw (9, 3) -- +(0, 1);

\foreach \x in {0, 4, 8}
{
	\draw[fill=black!15] (\x, 0) rectangle (\x + 2, 3);
	\draw (\x, 0) grid (\x + 2, 3);
	\foreach \i in {0,...,2}
	\foreach \j in {0,...,3}
	{
		\pgfmathsetmacro{\k}{\i+\j};
		\ifthenelse{\isodd{\k}}
		{\draw[fill=white] (\x + \i, \j) circle (2pt);}
		{\draw[fill=black] (\x + \i, \j) circle (2pt);}
	}
}

\foreach \x in {3, 7}
\foreach \y in {1, 3}
{
	\draw[fill=black] (\x, \y) circle (2pt);
}

\draw[fill=white] (11, 2) circle (2pt);
\draw[fill=white] (3, 0) circle (2pt);
\draw[fill=white] (7, 2) circle (2pt);
\draw[fill=white] (9, 4) circle (2pt);
\end{tikzpicture}
\end{subfigure}
\caption{\label{fig:family}Example of a family of sign-weighted compound graphs with respect to leaf $p$. The dashed line indicates an edge weight of $-1$. }
\end{figure}
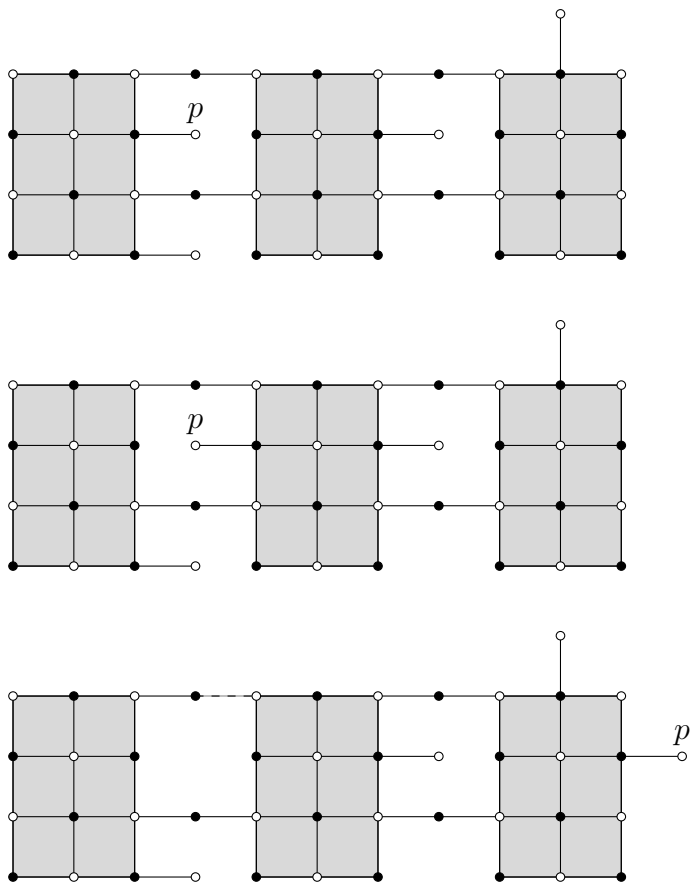

\begin{lemma}[Zero-sum lemma] Let $H_1, H_2, \hdots, H_k$ be all the members of a family. Then
$$\#\swh_1 \pm \#\swh_2 \pm \cdots \pm \#\swh_k = 0,$$
for some choice of signs.
\label{lem:zerosum}
\end{lemma}

The actual determination of signs seems to be rather complicated and depends on a number of factors, such as the relative positioning of the stems and leaves, and the choice of sign function for each sign-weighting. In any case, it is not important to the subsequent proofs.

\begin{proof} Since the supergraph is bipartite, we can ``color" its vertices with the numbers $1$ and $-1$. Let $\varphi$ be the map from the base vertices to the corresponding vertex of the supergraph. Then let $\sigma(v)$ be the function from the base vertices to $\{1, -1\}$ defined by taking the color of $\varphi(v)$. Thus $\sigma(v)$ will be the same for all vertices within a copy of the base graph, but different for vertices from neighboring copies. 

Next, for each base vertex $v$, define $s_v$ to be the function
$$s_v(v^\prime) = \begin{cases} \sigma(v^\prime) &\mbox{if } v^\prime \cong v \\ 0 &\mbox{otherwise.} \end{cases}$$
In more compact form, $s_v = \sum_{v^\prime \cong v} \sigma(v^\prime) e_{v^\prime}$, where $e_{v^\prime}$ denotes the elementary basis vector corresponding to $v'$. The function $s_v$ is defined on the vector space containing $v$, either $V_b$ or $V_w$ as the case may be.

Now, for each graph $\swh_i$, choose a sign function $\s_i$ and corresponding Kasteleyn matrix $\K i$. By Lemma ~\ref{lem:restriction}, we can choose the sign functions such that they agree on the base vertices. Thus, for any $i, j$ and base vertices $u, v$,
\begin{align*}
\K i_{uv} &= \begin{cases} \s_i(u, v) \w(u, v) &\mbox{if } u \sim v \\ 0 &\mbox{otherwise} \end{cases}\\
&= \begin{cases} \s_j(u, v) \w(u, v) &\mbox{if } u \sim v \\ 0 &\mbox{otherwise} \end{cases}\\
&= \K j_{uv}.
\end{align*}
If, on the other hand, $v$ is a base vertex and $x$ is a stem or leaf not equal to $p$, then
\begin{align*}
\K i_{vx} &= \begin{cases} \s_i(v, x) \s_i(v, x) &\mbox{if } v \sim x \\ 0 &\mbox{otherwise} \end{cases}\\
&= \begin{cases} 1 &\mbox{if } v \sim x \\ 0 &\mbox{otherwise} \end{cases}\\
&= \begin{cases} \s_j(v, x) \s_j(v, x) &\mbox{if } v \sim x \\ 0 &\mbox{otherwise} \end{cases}\\
&= \K j_{vx}.
\end{align*}
Thus, the matrices $\K i$ differ only in the column corresponding to vertex $p$. Let $K^\prime$ be the matrix formed by taking any of the $\K i$ and replacing the column corresponding to vertex $p$ with $s_q.$ Then, since $\K i e_p = \pm e_q,$ we have $K^\prime e_p = (\pm \K 1 \pm \K 2 \pm \cdots \pm \K k)e_p$. It follows that 
\begin{align*}
\det K^\prime &= \pm \det \K 1 \pm \det \K 2 \pm \cdots \pm \det \K k \\
&= \pm \#\swh_1 \pm \#\swh_2 \pm \cdots \pm \#\swh_k.
\end{align*}

We seek to show that $\det K^\prime = 0$. Let $S$ be the subspace of functions $f \in V_w$ such that $\sigma(x) f(x) = \sigma(y) f(y)$ whenever $x$ and $y$ are base vertices such that $x \cong y$, and $f(w) = 0$ if $w$ is a leaf. Analogously, let $T$ be the subspace of functions $f \in V_b$ such that $\sigma(x)f(x) = \sigma(y) f(y)$ whenever $x$ and $y$ are base vertices such that $x \cong y$, and $f(w) = 0$ if $w$ is a stem. Thus, in a sense, we may view $S$ and $T$ as the subspaces of $V_w$ and $V_b$ which respect the equivalence relation on the base vertices. We claim that $K^\prime S \subseteq T$. 

Let $f \in S$ be a function on the white vertices. Then, for any black base vertex $x$,
\begin{align*}
\sigma(x)K^\prime f(x) &= \sigma(x) \displaystyle\sum_{z \sim x} f(z) \s(z, x) \w(z, x) \\
&= \sum_{z \sim x} \sigma(z)f(z) \s(z, x) \w(z, x),
\end{align*}
since all vertices $z$ occuring in the sum must be base vertices, and consequently lie on the same copy of the base graph as $x$. (The terms corresponding to leaves disappear because $f$ assumes the value zero there.) Consider any other base vertex $y$ such that $x \cong y$. For every base vertex $z$ adjacent to $x$, there exists a corresponding base vertex $z^\prime \sim z$ neighboring $y$. Then $\s(z, x) = \s(z^\prime, y)$ and $\w(z, x) = \w(z^\prime, y)$, and by the condition that $f \in S$, we also have $\sigma(z)f(z) = \sigma(z^\prime)f(z^\prime).$ Therefore $\sigma(x)K^\prime f(x) = \sigma(y)K^\prime f(y)$ for any two base vertices $x$ and $y$ satisfying $x \cong y$.

Now let $w$ be a stem with neighbors $z_0$ and $z_0^\prime$. Then we have 
\begin{align*}
K^\prime f(w) &= \displaystyle\sum_{z \sim w} f(z) \s(z, w) \w(z, w) \\
&= f(z_0) \s(z_0, w) \w(z_0, w) + f(z_0^\prime) \s(z_0^\prime, w) \w(z_0^\prime, w).
\end{align*}
Because of the sign-weighting procedure, we know that $$\s(z_0, w) \w(z_0, w) = \s(z_0^\prime, w) \w(z_0^\prime, w) = 1.$$ Since $z_0$ and $z_0^\prime$ are on adjacent copies of the base graph, $f(z_0) = -f(z_0^\prime).$ So $K^\prime f(w) = 0$ for all stems $w$. Thus $K^\prime S \subseteq T$. 

Now, consider the subspace $S^\prime$ formed by taking the direct sum of $S$ and $e_p$. Since $K^\prime e_p = s_q \in T$, we have that $K^\prime S^\prime \subseteq T$. Note that the dimension of $S^\prime$ is $N + 1$, where $N$ is the number of white vertices, having as its basis the set of vectors $s_v$ (for white base vertices $v$) and $e_p$. The dimension of $T$ is $N$. It follows that $K^\prime$ is singular, and thus $\det K^\prime = 0$ as desired.
\end{proof}

\begin{theorem}[Divisibility Theorem]
Let $\swh$ be a sign-weighted compound graph and let $G$ be its base graph. Then
$$\#G \mid \#\swh.$$
\end{theorem}

\begin{proof}
The theorem follows from the zero-sum lemma by induction. Pick any copy $G$ of the base graph. We will induct on the number of leaves attached to $G$.

The base case is that there are no leaves. Then, in any matching all of the stems are forced to match outwards, or else there would be an unequal number of black and white vertices remaining inside $G$. So in this case the number of matchings is divisible by $\#G$. 

Now for the inductive step, suppose $k$ leaves are attached to $G$. Pick any of these leaves; this leaf defines a family of compound graphs. In every sibling, this leaf is attached to a different copy of the base graph, leaving only $k-1$ leaves attached to $G$. Thus, by the inductive hypothesis, the number of matchings of every sibling is divisible by $\#G$. It follows from the lemma that $\#\swh$ is divisible by $\#G$ as well. This completes the proof.
\end{proof}

We note that the proof of the divisibility theorem can in fact be rephrased directly in terms of linear algebra, using a sequence of column operations. However, this proof gives no more insight than the combinatorial proof presented above, so we omit it here.
\end{section}

\begin{section}{Application to Rectangles}

The primary application of the divisibility theorem is that it provides a proof of the divisibility property of rectangles. Let $R(m, n)$ denote the $m \times n$ rectangular graph, and index the rows and columns from 1 to $m$ and from 1 to $n$, respectively. Then the vertex $(i, j)$ is the vertex in the $i$th row and $j$th column, and $(1, 1)$ is the upper left corner. The edges all have weight 1.

\begin{theorem}[Divisibility Property]\label{thm:rectangles}
If $a + 1 \mid A + 1$ and $b + 1 \mid B + 1$, then
$$\#R(a, b) \mid \#R(A, B).$$
\end{theorem}

Note that if either $a$ and $b$ are both odd, or $A$ and $B$ are both odd, then the theorem follows trivially because a rectangle of odd area has no matchings. Therefore we may assume that both $ab$ and $AB$ are even, which implies that both $R(a, b)$ and $R(A, B)$ are in $\G$.

We first present the following lemma.

\begin{lemma}\label{lem:h=swh}
Let $H$ be a compound graph. If in every interior face of the reduced graph there are an odd number of leaves, then $\#H = \#\swh$, i.e., the edge-weights are unchanged by sign-weighting.
\end{lemma}
\begin{proof}[Proof of Lemma~\ref{lem:h=swh}]
We want to show that the function $f \equiv 1$ is a valid sign function on the reduced graph. To prove this, we must simply show that every face has 2 mod 4 edges. Now every face of the reduced graph consists of a simple cycle together with the leaves in its interior. The edges of the simple cycle can be partitioned into pairs according to which stem they are attached to. Since the supergraph is bipartite, there are an even number of stems in the cycle. Thus the cycle of the reduced graph must have 0 mod 4 edges. Because the number of leaves is odd, the leaves contribute 2 mod 4 edges, so the number of edges of the face is 2 mod 4.
\end{proof}

\begin{figure}[htb]
\centering
\begin{tikzpicture}[scale=.7]

\foreach \x in {0, 5, 10}
\foreach \y in {0, 6}
\draw[fill=black!15] (\x, \y) rectangle +(3, 4);

\draw (0, 0) grid (13, 10);

\foreach \x in {0,...,12}
\draw[white, very thick] (\x, 5) -- +(1, 0);

\foreach \x in {4, 9}
\foreach \y in {0,...,9}
\draw[white, very thick] (\x, \y) -- +(0, 1);

\draw[white, very thick] (3, 6) -- (5, 6) (10, 6) -- (10, 4);

\foreach \p in {(3, 10), (4, 8), (3, 4), (4, 2), (4, 0), (8, 9), (8, 7), (9, 3), (8, 1)}
\draw[white, very thick] \p -- +(1, 0);

\foreach \p in {(1, 5), (3, 4), (5, 4), (7, 4), (9, 5), (11, 5)}
\draw[white, very thick] \p -- +(0, 1);

\foreach \x in {0,...,13}
\foreach \y in {0,...,10} 
{
	\pgfmathsetmacro{\z}{\x+\y};
	\ifthenelse{\isodd{\z}}
	{
		\ifthenelse{\equal{(\x, \y)}{(10, 5)} \OR \equal{(\x, \y)}{(4, 5)}}
		{}{\draw[fill=black] (\x, \y) circle (2pt);}
	}
	{
		\ifthenelse{\equal{(\x, \y)}{(4, 6)} \OR \equal{(\x, \y)}{(9, 5)}}
		{}{\draw[fill=white] (\x, \y) circle (2pt);}
	}
}

\end{tikzpicture}
\caption{\label{fig:recdiv}A compound graph $\Rd$ with base graph $R(5, 4)$}
\end{figure}
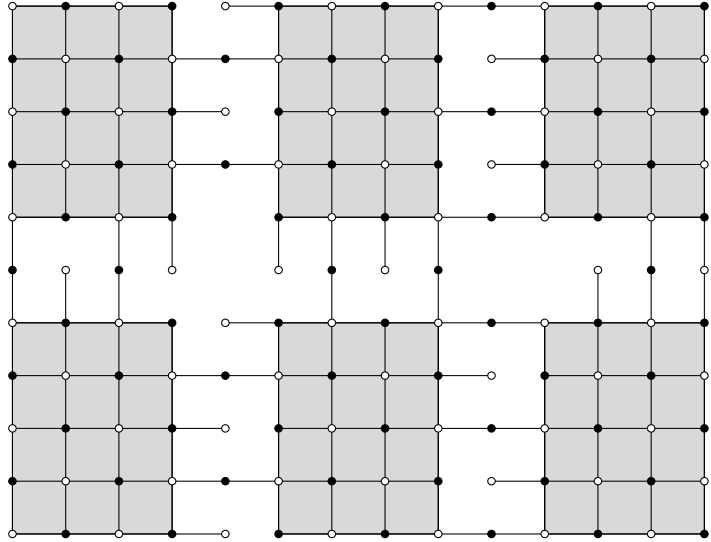

\begin{proof}[Proof of Theorem~\ref{thm:rectangles}]

We will show that the number of matchings of $R(A, B)$ is the sum of the number of matchings of many compound graphs, all of which have $R(a, b)$ as their base graph and $R(\frac{A + 1}{a + 1}, \frac{B + 1}{b + 1})$ as their supergraph. Henceforth, $R(A, B)$ will be simply written $R$. 

Let $A + 1 = k(a + 1)$ and $B + 1 = l(b + 1).$ Then $R$ can be partitioned into $kl$ copies of $R(a, b)$, where the upper left vertex of each copy is of the form $(1 + i(a + 1), 1 + j(b + 1))$, together with the remaining vertices which we call ``roots". 
Call a domino (an edge that appears in a matching) whose vertices are both roots a ``root domino". 

Given any set $S$ of pairwise disjoint root dominoes, define $R^\prime$ to be the induced graph formed from $R$ by deleting all the endpoints of dominoes in $S$, along with all edges between roots. Then the set of matchings of $R$ whose root dominoes are precisely the dominoes in $S$ is equal to the set of matchings of $R^\prime$. 

Next, note that the remaining roots have degree 2, and thus can match in one of two directions. Let $D$ be a set of edges representing a choice of direction for all of the white roots. Then, given $D$, define a new modified graph $\Rd$ as follows: for each white root, delete the edge incident to it which is not in $D$. Thus the set of matchings of $R^\prime$ whose white roots match as specified by $D$ now corresponds to the set of matchings of $\Rd$. So
\begin{align*}
\#R &= \displaystyle\sum_S R^\prime \\
&= \sum_S \sum_D \#\Rd.
\end{align*}

We now have that $\Rd$ is a compound graph. An example is pictured in Figure~\ref{fig:recdiv}. The black roots are now the stems, and the white roots are the leaves. One condition remains to be verified before we can apply the divisibility theorem: because we wish to count unweighted matchings, we need to be sure that the process of sign-weighting does not change the stem and leaf weights.

In order to apply the lemma, we must simply show that in every face of the reduced graph of $\Rd$, there are an odd number of leaves. We prove this by induction on the number of root dominoes. For any matching, every vertex of the form $(i(a + 1), j(b + 1))$ must necessarily be included by a root domino, because all four of its neighbors are roots. Thus the base case is that there are $(k - 1)(l - 1)$ root dominoes in $S$. If this is the case, it is easy to see by inspection, as pictured in Figure~\ref{fig:recdiv}, that all faces have an odd number of white leaves. The faces which lie between two copies of the base graph contain exactly one white leaf, and the faces which lie between four copies of the base graph contain either one or three white leaves, depending on whether the central vertex was black or white. (Special care must be taken in the case that $a$ or $b$ equals 2, but we leave this minor case to the reader.)

Suppose there are more than $(k - 1)(l - 1)$ dominoes in $S$. Pick one which does not have an endpoint of the form $(i(a + 1), j(b + 1))$. If we let $S^\prime$ be the set formed by removing this domino from $S$, then by the inductive hypothesis the corresponding graph $\Rd(S^\prime)$ has an odd number of leaves in all the faces of its reduced graph. Now $\Rd(S)$ is derived from $\Rd(S^\prime)$ be deleting a pair of black and white vertices. By deleting the black vertex, we make one face out of two faces which were formerly separated by a black stem. If one of these faces was the outer face, then both become the outer face and we are done. If both faces were inner faces, then by the inductive hypothesis both had an odd number of leaves, so the combined number of leaves is even. However, deleting the white vertex removes one of these leaves, leaving an odd number of leaves in the resulting face. Since the rest of the faces are left unchanged, every face of the reduced graph of $\Rd(S)$ also has an odd number of leaves, completing the induction.

Thus the graph $\Rd$ obeys the conditions of the lemma, so the edge-weights are preserved by sign-weighting. By the divisibility theorem, $\#\Rd$ is divisible by $\#R(a, b)$. Therefore, $\#R(A, B) = \sum_S \sum_D \#\Rd$ is divisible by $\#R(a, b)$, as desired.
\end{proof}
\end{section}

In a similar manner, the divisibility theorem can be applied to subgraphs of the hexagonal lattice and other lattices. However, because application of Lemma~\ref{lem:h=swh} essentially requires stems and leaves to alternate black and white, there are strong restrictions on which lines of symmetry can be chosen.

We briefly mention one other interesting numerical property of rectangles. The sequence $\#R(k, n)$, for fixed $k$, is a linearly recursive sequence. Intriguingly, if this recurrence is extended backwards to negative $n$, then $|\#R(k, -1 - n)| = |\#R(k, -1 + n)|$. Propp provides a combinatorial proof of this ``reciprocity" in \cite{propp01}. An analogous property also turns up in the matchings of other regions, such as Aztec diamonds and Aztec pillows. It is not known whether this property is somehow related to the divisibility property.

\begin{section}{Ciucu's Factorization Theorem}

The same general approach taken towards proving the divisibility theorem can be applied to provide an alternative proof of Ciucu's factorization theorem, when the graph in question is bipartite. The proof is strikingly similar, but we will be careful to point out the subtle differences. The technique of showing that a certain matrix has zero determinant is applied to prove Ciucu's lemma, which is analogous to the zero-sum lemma, and from there Ciucu's factorization theorem follows in much the same way that the divisibility theorem follows from the zero-sum lemma. To clarify, by ``Ciucu's lemma" we refer to Lemma 1.1 of \cite{ciucu97}, whereas by ``Ciucu's factorization theorem" we refer to Theorem 1.2 of the same paper. Both are restated below for the reader's convenience in the bipartite case.

Using the terminology defined in Section 2, Ciucu's lemma essentially involves compound graphs whose supergraph is $P_2$. Let $G_1$ and $G_2$ be the two copies of the base graph, where $G_2$ is the reflection of $G_1$ across a line $\ell$ exterior to both $G_1$ and $G_2$. Then connect $G_1$ and $G_2$ with $w$ stems and $w$ leaves, lying on line $\ell$, where $w$ can be any positive integer. When read along line $\ell$, the vertices must alternate stems and leaves. Moreover, the total number of black and white vertices of the resulting graph $H$ must be equal, so that $H \in \G$. However, the zero-sum lemma cannot be applied because the stems may not necessarily be black, the leaves may not necessarily be white, and $G$ may not necessarily belong to $\G$. An example of such a compound graph is shown in Figure~\ref{fig:ciucu}.

We will, as before, consider weighted matchings of $H$, where the weighting respects the equivalence relation on the base edges. However, there will be no sign-weighting; the weights on all the stems and leaves will remain 1. We will also define a family of compound graphs analogously; in this case, each famiy will contain two graphs.

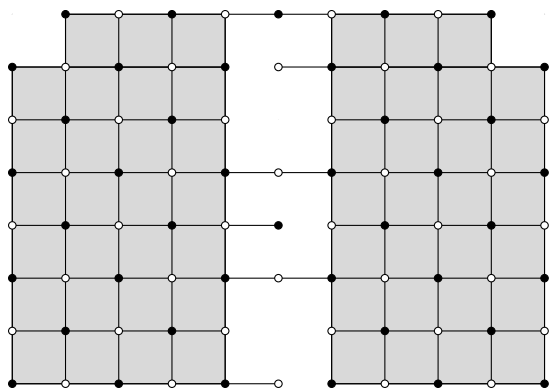
\begin{figure}[htb]
\centering
\begin{tikzpicture}[scale=.7]

\draw[fill=black!15] (0, 0) rectangle +(4, 6);
\draw[fill=black!15] (6, 0) rectangle +(4, 6);
\draw[fill=black!15] (1, 6) rectangle +(3, 1);
\draw[fill=black!15] (6, 6) rectangle +(3, 1);

\draw (0, 0) grid (10, 7);

\foreach \p in {(0, 7), (9, 7), (4, 6), (4, 5), (5, 5), (5, 3), (4, 1), (5, 1), (5, 0)}
\draw[white, very thick] \p -- +(1, 0);

\foreach \p in {(0, 6), (10, 6), (5, 0), (5, 1), (5, 2), (5, 3), (5, 4), (5, 5), (5, 6)}
\draw[white, very thick] \p -- +(0, 1);

\foreach \x in {0,...,10}
\foreach \y in {0,...,7} 
{
	\pgfmathsetmacro{\z}{\x+\y};
	\ifthenelse{\isodd{\z}}
	{
		\ifthenelse{\equal{(\x, \y)}{(0, 7)} \OR \equal{(\x, \y)}{(10, 7)}}
		{}{\draw[fill=white] (\x, \y) circle (2pt);}
	}
	{
		\ifthenelse{\equal{(\x, \y)}{(5, 1)} \OR \equal{(\x, \y)}{(5, 5)}}
		{}{\draw[fill=black] (\x, \y) circle (2pt);}
	}
}

\end{tikzpicture}
\caption{\label{fig:ciucu}A compound graph with both black and white stems and leaves}
\end{figure}

\begin{lemma}[Ciucu's Lemma] \label{lem:ciucu}
Let $H_1$ and $H_2$ be the two members of a family. Then
$$\#H_1 - \#H_2 = 0.$$
\end{lemma}

As before, we must first prove a minor lemma on sign functions.

\begin{lemma}
There exists a sign function $\s$ on $H$ that respects the equivalence relation and additionally satisfies the following property: For all stems $p$, if $p$ is incident to edges $e_1$ and $e_2$, then $\s e_1 = \s e_2$ if $p$ is black and $\s e_1 = -\s e_2$ if $p$ is white.
\label{lem:ciucu sgn}
\end{lemma}

\begin{proof}[Proof of Lemma ~\ref{lem:ciucu sgn}]
By Lemma~\ref{lem:sgn exists}, there exists a sign function $f$ on the base graph $G$. Now define a sign function $\s$ on the compound graph $H$ by setting $\s$ in accordance with $f$ on the base edges, and setting $\s$ in accordance with the conditions of the lemma on the edges connected to stems. 

Because of the condition that the stems and leaves must alternate along line $\ell$, each face of the reduced graph $R$ is 6-sided. Consider a face $F_R$ of the reduced graph and corresponding face $F_H$ of $H$. Suppose that $F_H$ is surrounded by the two stems $p$ and $p^\prime$. By the same logic as in the proof of Lemma~\ref{lem:restriction}, we have
\begin{align*}
\prod_{e \in F_H} \s e &= \prod_{e \in F_R} \s e \\
&= \begin{cases} 1 &\mbox{if } p \text{ and } p^\prime \text{ are the same color,} \\ -1 &\mbox{if } p \text{ and } p^\prime \text{ are different colors.} \end{cases}
\end{align*}

Now every face of the reduced graph has 6 edges. The additional edges added in $F_H$ are the two paths connecting the neighbors of $p$ and $p^\prime$. If $p$ and $p^\prime$ are the same color, each of these paths has even length, contributing 0 mod 4 edges; otherwise, each path has odd length, contributing 2 mod 4 edges. In either case $F_H$ is positive. Since every face of the base graph is positive as well, $\s$ is a valid sign function.
\end{proof}

\begin{proof}[Proof of Ciucu's lemma]
We can assume, without loss of generality, that the leaf $p$ defining the family is white; the other case follows by symmetry. Let $K_1$ and $K_2$ be the Kasteleyn operators of $H_1$ and $H_2$, respectively. Because leaf $p$ remains in the same face in both graphs, the same sign function is valid for both $H_1$ and $H_2$ (and thus there is no need for sign-weighting). Thus, $K_1$ and $K_2$ are identical in all columns except column $p$. Let $K^\prime$ be the matrix formed by taking $K_1$ or $K_2$ and replacing the column corresponding to vertex $p$ with $s_q$, as defined in Lemma ~\ref{lem:zerosum}, where $q$ is the vertex adjacent to $p$. Then, as before, $\det K^\prime = \pm \det K_1 \pm \det K_2$. We wish to show that $\det K^\prime = 0$. 

Define $S$ to be the subspace of functions $f \in V_w$ such that $\sigma(x) f(x) = \sigma(y) f(y)$ whenever $x$ and $y$ are base vertices such that $x \cong y$, where $\sigma$ is defined as before, and $f(w) = 0$ whenever $w$ is a leaf (with no restrictions if $w$ is a stem). Likewise, define $T$ to be the subspace of functions $f \in V_b$ such that $\sigma(x) f(x) = \sigma(y) f(y)$ whenever $x$ and $y$ are base vertices such that $x \cong y$, and $f(w) = 0$ whenever $w$ is a stem (with no restrictions if $w$ is a leaf). We claim that $K^\prime S \subseteq T$. 

Now, if $x$ is a black base vertex that is not adjacent to a white stem, then for every function $f \in S$,
$$\sigma(x) K^\prime f(x) = \sigma(x) \displaystyle\sum_{z \sim x} f(z) \s(z, x) \w(z, x).$$
Then, for all $y \cong x$, we have $\sigma(x) K^\prime f(x) = \sigma(y) K^\prime f(y)$ by the same argument given in the proof of the zero-sum lemma.

If $x$ is a black base vertex adjacent to a white stem $w$, then 
$$\sigma(x) K^\prime f(x) = \sigma(x) f(w) \s(w, x) \w(w, x) + \sigma(x) \displaystyle\sum_{z \in B} f(z) \s(z, x) \w(z, x),$$
where $B$ is the set of base vertices adjacent to $x$. Let $y$ be another base vertex such that $y \cong x,$ and let $B^\prime$ be the set of base vertices adjacent to $y$. Since there are only two copies of the base graph, $y$ must also be adjacent to $w$. So
\begin{align*}
\sigma(y) K^\prime f(y) &= \sigma(y) f(w) \s(w, y) \w(w, y) + \sigma(y) \sum_{z^\prime \in B^\prime} f(z^\prime) \s(z^\prime, y) \w(z^\prime, y) \\
&= - \sigma(y) f(w) \s(w, x) \w(w, x) +\sum_{z^\prime \in B^\prime} \sigma(z^\prime)f(z^\prime) \s(z^\prime, y) \w(z^\prime, y) \\
&= \sigma(x) f(w) \s(w, x) \w(w, x) + \displaystyle\sum_{z \in B} \sigma(z) f(z) \s(z, x) \w(z, x) \\
&= \sigma(x) K^\prime f(x),
\end{align*}
and thus the condition still holds.

Finally, if $w$ is a black stem, then $w$ is adjacent to two white vertices $z_1$ and $z_2$, and since $z_1 \cong z_2$, 
\begin{align*}
K^\prime f(w) &= \displaystyle\sum_{z \sim w} f(z) \s(z, w) \w(z, w) \\
&= f(z_1) \s(z_1, w) \w(z_1, w) + f(z_2) \s(z_2, w) \w(z_2, w) \\
&= f(z_1) \s(z_1, w) \w(z_1, w) - f(z_1) \s(z_1, w) \w(z_1, w) \\
&= 0.
\end{align*}

This concludes the proof that $K^\prime S \subseteq T$. Since $K^\prime e_p = s_q \in T$ as well, $K^\prime$ takes the $(N + 1)$-dimensional space $S \oplus e_p$ to the $N$-dimensional space $T$, showing that $K^\prime$ is singular. Thus $\#H_1 \pm \#H_2 = 0$. Now the sign of the operator $\pm$ is independent of the weighting, so by choosing a positive weighting we can deduce that we must have $\#H_1 - \#H_2 = 0,$ as desired. 
\end{proof}

Ciucu's lemma can in fact be slightly generalized by allowing sign-weighting, so that the stems and leaves need not necessarily alternate. We have omitted this generalization, however, as it does not appear to offer much new insight. 

Ciucu's theorem uses Ciucu's lemma to count matchings in the same way that we used the zero-sum lemma to count matchings of rectangles. The setup is similar to before, except there are no leaves. Let $G_1$ and $G_2$ be two copies of a planar bipartite graph $G$, where $G_2$ is the reflection of $G_1$ across a line $\ell$ exterior to both $G_1$ and $G_2$. Construct $2w$ stems on $\ell$, such that the resulting graph $H$ (consisting of $G_1$ and $G_2$ together with these $2w$ vertices) has an equal number of black and white vertices. 

\begin{theorem}[Ciucu's theorem]
Let the stems of $H$ be labeled ``even" or ``odd," alternatingly. Define $G_{bw}$ to be the induced subgraph of $G$ formed by deleting from $G$ every vertex adjacent to a black even vertex or a white odd vertex in $H$, and define $G_{wb}$ to be the induced subgraph formed by deleting every vertex adjacent to a white even vertex or a black odd vertex in $H$. Then $$\#H = 2^w \#G_{bw} \#G_{wb}.$$
\end{theorem}

\begin{proof}
The proof of Ciucu's theorem follows from Ciucu's lemma in a very similar manner to the way the divisibility theorem follows from the zero-sum lemma. We refer the reader to Ciucu's paper \cite{ciucu97} for details. In the statement presented the vertices on $\ell$ are assumed to be independent and have degree 2; Ciucu shows that the general case is easily deduced from this case via a graph transformation.
\end{proof}

It is striking that the proof of Ciucu's lemma using linear algebra and the original proof involving 2-factors (which has the advantage of not assuming bipartiteness) seem entirely different. An interesting question to ask is whether there is a combinatorial proof of the zero-sum lemma in the style of 2-factors. Perhaps there exists a more general statement which generalizes both Ciucu's factorization theorem and the divisibility theorem.
\end{section}

\subsection*{Acknowledgements}
\thanks{I was supported by MIT's Undergraduate Research Opportunities Program in the summer of 2013, under the mentorship of Professor Henry Cohn. I would like to thank Henry Cohn for many stimulating discussions.}

\bibliographystyle{plain}

\end{document}